\documentclass{amsart}

\usepackage{enumerate}

\usepackage{amssymb,latexsym,amsxtra,amscd,amsfonts}

\usepackage[mathscr]{eucal}

\usepackage{dsfont}

\usepackage{calrsfs}

\newtheorem{thm}{Theorem}[section]

\newtheorem{cor}[thm]{Corollary}

\newtheorem{lem}[thm]{Lemma}

\newtheorem{pr}[thm]{Proposition}

\theoremstyle{definition}

\newtheorem*{convention}{Convention}

\theoremstyle{remark}

\newcounter{obr}[section]

\newcounter{pvv}[section]

\renewcommand{\thepvv}{\thesection.\arabic{pvv}}

\newcommand{\saM}{\m_{sa}}

\newcommand{\psaM}{\m_{*sa}}

\newcommand{\A}{\mathcal A}

\newcommand{\m}{\mathcal M}

\newcommand{\N}{\mathds N}

\newcommand{\nn}{\ensuremath{\mathbb N}}

\newcommand{\zz}{\ensuremath{\mathbb Z}}

\newcommand{\cc}{\ensuremath{\mathbb C}}

\newcommand{\rr}{\ensuremath{\mathbb{R}}}

\newcommand{\Q}{\ensuremath{\mathbb{Q}}}

\def\Re{\mathop{\rm Re}\nolimits}

\def\Im{\mathop{\rm Im}\nolimits}

\def\sp{\mathop{\rm sp}\nolimits}

\def\dom{\mathop{\rm dom}\nolimits}

\def\rng{\mathop{\rm rng}\nolimits}

\newcommand{\vNa}{von Neumann algebra}

\newcommand{\Ca}{$C^\ast$-algebra}

\newcommand{\JBWs}{$JBW^\ast$-algebra}

\newcommand{\JBW}{$JBW$-algebra}

\title{Decompositions of preduals of JBW and JBW$^*$ algebras}

\author{Martin Bohata, Jan Hamhalter and Ond\v{r}ej F.K. Kalenda}

\address{Czech Technical University in Prague, Faculty of Electrical Engineering, Department of Mathematics, Technick\'a 2, 166~27 Prague 6, Czech Republic}

\email{bohata@math.feld.cvut.cz}

\email{hamhalte@math.feld.cvut.cz}

\address{Charles University in Prague, Faculty of Mathematics and Physics, Department of Mathematical Analysis,
Sokolovsk\'a 86, 186~75 Praha 8, Czech Republic}

\email{kalenda@karlin.mff.cuni.cz}

\begin{document}

\begin{abstract}
We prove that the predual of any JBW$^*$-algebra is a complex $1$-Plichko space and the predual of any JBW-algebra is a real $1$-Plichko space. I.e., any such space has a countably $1$-norming Markushevich basis, or, equivalently, a commutative $1$-projectional skeleton. This extends recent results of the authors who proved the same for preduals of von Neumann algebras and their self-adjoint parts. However, the more general setting of Jordan algebras turned to be much more complicated.
We use in the proof a set-theoretical method of elementary submodels. As a byproduct we obtain a result on amalgamation of projectional skeletons.
\end{abstract}




\thanks{Our research was supported in part by the grant GA\v{C}R P201/12/0290.}

\maketitle

\section{Introduction and main results}

The aim of the present paper is to show that the predual of any \JBW{} is $1$-Plichko (i.e., it has a countably $1$-norming Markushevich basis or, equivalently, it admits a commutative $1$-projectional skeleton) and the same holds also for preduals of \JBWs{}s. This extends previous results of the authors who showed in \cite{BHK} the same statements on preduals of \vNa{}s and their self-adjoint parts. \JBWs{}s can be viewed as a generalization of \vNa{}s, this class was introduced and studied in \cite{edwards}; a \JBW{} can be represented as the self-adjoint part of a \JBWs{} (see \cite{edwards}). Precise definitions and a necessary background on these algebras is given in Section~\ref{S:Jordan} below.

$1$-Plichko spaces form one of the largest classes of Banach spaces which admit a reasonable decomposition to separable pieces.
This class and some related classes of Banach spaces together with the associated classes of compact spaces were thoroughly
 studied for example in \cite{val90,val91,survey}. The class of $1$-Plichko spaces can be viewed as a common roof of previously studied classes of weakly compactly generated spaces \cite{AL}, weakly $K$-analytic Banach spaces \cite{talagrand}, weakly countably determined (Va\v{s}\'ak) spaces \cite{vasak,mer87} and weakly Lindel\"of determined spaces \cite{AM}. Examples of 
 $1$-Plichko spaces include $L^1$ spaces, order continuous Banach lattices, spaces $C(G)$ for a compact abelian group $G$
\cite{val-exa}; preduals of \vNa{}s and their self-adjoint parts \cite{BHK}.

Let us continue by defining $1$-Plichko spaces and some related classes. We will do it using the notion of a projectional 
skeleton introduced in \cite{kubisSkeleton}. If $X$ is a Banach space, a \emph{projectional skeleton} on $X$ is an indexed system of 
bounded linear projections $(P_\lambda)_{\lambda\in\Lambda}$ where $\Lambda$ is an up-directed set such that the following 
conditions are satisfied:

	 \begin{itemize}

	 \item[(i)] $\sup_{\lambda\in\Lambda}\|P_\lambda\|<\infty$,

	\item[(ii)] $P_\lambda X$ is separable for each $\lambda$,

	\item[(iii)] $P_\lambda P_\mu=P_\mu P_\lambda =P_\lambda$ whenever $\lambda\le\mu$,

	\item[(iv)] if $(\lambda_n)$ is an increasing sequence in $\Lambda$, it has a supremum $\lambda\in\Lambda$ and $P_\lambda [X]=\overline{\bigcup_n P_{\lambda_n}[X]}$,

		\item[(v)]$X=\bigcup_{\lambda\in\Lambda} P_\lambda [X]$.

	 \end{itemize}

\noindent The subspace $D=\bigcup_{\lambda\in\Lambda} P^*_\lambda [X^*]$ is called the \emph{subspace induced by the skeleton}.	 
If $\|P_\lambda\|=1$ for each $\lambda\in\Lambda$, the family $(P_\lambda)_{\lambda\in\Lambda}$ is said to be \emph{$1$-projectional skeleton}. 
The skeleton $(P_\lambda)_{\lambda\in\Lambda}$ is said to be \emph{commutative} if
$P_\lambda P_\mu=P_\mu P_\lambda$ for any $\lambda,\mu\in \Lambda$. 
A Banach space having a commutative ($1$-)projectional skeleton is called \emph{($1$-)Plichko}. 

This is not the original definition used in \cite{survey,val-exa} which says that $X$ is ($1$-)Plichko if $X^*$ admits a ($1$-)norming $\Sigma$-subspace. Let us recall that a subspace $D\subset X^*$ is $r$-norming ($r\ge 0$) if the formula
$$ |x|=\sup\{|x^*(x)| : x^*\in D, \|x^*\|\le1 \} $$
defines an equivalent norm on $X$ for which $\|\cdot \|\le r |\cdot |$.

Further, a subspace $D\subset X^*$ is a \emph{$\Sigma$-subspace} of $X^*$ if there is a linearly dense set $M\subset X$ such that
$$D=\{ x^*\in X^*: \{m\in M: x^*(m)\ne 0\}\mbox{ is countable}\}.$$
It follows from \cite[Proposition 21 and Theorem 27]{kubisSkeleton} that a norming subspace of $X^*$ is a $\Sigma$-subspace of $X^*$ if and only if it is induced by a commutative projectional skeleton, therefore our definitions are equivalent to the original ones.

Finally, recall that a Banach space $X$ is called \emph{weakly Lindel\"of determined} (shortly \emph{WLD}) if $X^*$ is a $\Sigma$-subspace of itself or, equivalently, if $X^*$ is induced by a commutative projectional skeleton in $X$.

Now we can formulate our main results. The following theorem extends \cite[Theorems 1.1 and 1.4]{BHK} to the more general setting of Jordan algebras. Precise definitions of the respective algebras are in the following section.

\begin{thm}\label{T:main} \

\begin{itemize}

	\item Let $\m$ be any \JBWs. Its predual $\m_*$ is a (complex) $1$-Plichko space. Moreover, $\m_*$ is WLD if and only if $\m$ is $\sigma$-finite. In this case it is even weakly compactly generated.

	\item Let $\m$ be any \JBW. Its predual $\m_*$ is a (real) $1$-Plichko space. Moreover, $\m_*$ is WLD if and only if $\m$ is $\sigma$-finite. In this case it is even weakly compactly generated.

\end{itemize}

\end{thm}

As a corollary we get the following extension of a result of U.~Haagerup \cite[Theorem IX.1]{GGMS} on preduals of \vNa{}s. It follows immediately from Theorem~\ref{T:main} and the definition of projectional skeletons.
A Banach space $X$ is said to have \emph{separable complementation property} if each countable subset of $X$ is contained in some separable complemented subspace of $X$.

\begin{cor} \

\begin{itemize}

	\item The predual of any \JBWs{} enjoys the separable complementation property.

	\item The predual of any \JBW{} enjoys the separable complementation property.

\end{itemize}

\end{cor}

Since the bidual of any $JB$-algebra is a \JBW{} and the bidual of any $JB^\ast$-algebra is a \JBWs{},
the following result follows.

\begin{cor} \

\begin{itemize}

	\item The dual of any $JB^\ast$-algebra is a (complex) $1$-Plichko space.

	\item The dual of any $JB$-algebra is a (real) $1$-Plichko space.

\end{itemize}

\end{cor}

The rest of the paper is devoted to the proof of Theorem~\ref{T:main}. The proof uses some ideas from \cite{BHK}
but is much more involved. As a byproduct we obtain the following theorem which seems to be of an independent interest.

\begin{thm}\label{T:lepeni} 

Let $X$ be a (real or complex) Banach space. Suppose that there is an indexed family $(R_\lambda)_{\lambda\in\Lambda}$ 
of linear projections on $X$ such that the following assertions are satisfied.

\begin{itemize}

	\item[(i)] $\sup_{\lambda\in\Lambda}\|R_\lambda\|<\infty$.

	\item[(ii)] $R_\lambda [X]$ is WLD for each $\lambda\in\Lambda$.

	\item[(iii)] If $\lambda,\mu\in\Lambda$ are such that $\lambda\le\mu$, then $R_\lambda R_\mu=R_\mu R_\lambda=R_\lambda$.

	\item[(iv)] If $\lambda_1\le \lambda_2\le \dots$ are in $\Lambda$, then $\lambda=\sup_n \lambda_n$ exists in $\Lambda$ and, moreover $R_\lambda [X]=\overline{\bigcup_n R_{\lambda_n}[X]}$.

	\item[(v)] $X=\bigcup_{\lambda\in\Lambda} R_\lambda [X]$.

\end{itemize}
Then there is a projectional skeleton on $X$ such that the subspace of $X^*$ induced by the skeleton equals $\bigcup_{\lambda\in\Lambda}R_\lambda^*[X^*]$. 
\end{thm}

This theorem says, roughly speaking, that if $X$ admits a ``projectional skeleton'' from projections whose ranges are just WLD (not necessarily separable), then $X$ has also a ``proper'' projectional skeleton inducing the same subspace of the dual. 
We do not know whether the same holds for commutative skeletons.

The rest of the paper is organized as follows: In Section~\ref{S:Jordan} we collect some basic facts on Jordan Banach algebras and their important subclasses. Section~\ref{S:proj} is devoted to projections in \JBWs{}s. The main purpose of that section is to prove Propositions~\ref{P:skel1} and~\ref{P:skel2}. They are the first step towards a proof of Theorem~\ref{T:main} and roughly say that in the respective preduals there are families of projections satisfying the assumptions of Theorem~\ref{T:lepeni}. Section~\ref{S:elem} contains a brief exposition of the method of elementary submodels and several
auxiliary results needed later. In the last section the method of elementary submodels is used to prove Theorem~\ref{T:lepeni}
and finally Theorem~\ref{T:main}.

Our notation is mostly standard. We only point out that for a mapping $f$ we distinguish $f(x)$ -- the value of $f$ at $x$ --
and $f[A]$ -- the image of the set $A$ under the mapping $f$. This distinction is necessary due to the use of set-theoretical tools.

\section{Jordan Banach algebras}\label{S:Jordan}

In this section we collect basic definitions and properties of Jordan algebras which are needed in the formulations and
proofs of our results. We use namely the books \cite{joa,AS,nonasoc} and the paper \cite{edwards}.

A \emph{Jordan algebra} is a real or complex algebra $\A=(\A,+,\circ)$, non-associative in general, which satisfies moreover the following two axioms:

\begin{itemize}

	\item $x\circ y=y\circ x$ for $x,y\in\A$,

	\item $(x\circ x)\circ(x\circ y)=x\circ(y\circ(x\circ x))$ for $x,y\in\A$.

\end{itemize}
If $\A=(\A,+,\cdot)$ is an associative algebra, the \emph{special Jordan product} on $\A$ is defined by $x\circ y=\frac12(x\cdot y+y\cdot x)$. Then $(\A,+,\circ)$ is a Jordan algebra. A \emph{Jordan subalgebra of $\A$} is a subalgebra of $(\A,+,\circ)$, i.e.\ a linear subspace of $\A$ closed under the special Jordan product.
Any algebra isomorphic to a Jordan subalgebra of an associative algebra  is called a \emph{special Jordan algebra}.
We will use several times the Shirshov-Cohn theorem \cite[Theorem 2.4.14]{joa} which says that any Jordan algebra generated by two elements (and $1$ if it is unital) is special.

An important further operation in Jordan algebras is the \emph{Jordan triple product} defined by the formula
$$\{xyz\}=(x\circ y)\circ z+x\circ(y\circ z)-(x\circ z)\circ y,\quad x,y,z\in\A.$$
A \emph{Jordan Banach algebra} is a real or complex Jordan algebra $\A$ equipped with a complete norm satisfying
$$\|x\circ y\|\le\|x\|\cdot\|y\|\mbox{ for }x,y\in \A.$$

A \emph{$JB$-algebra} is a real Jordan Banach algebra $\A$ satisfying moreover the following two axioms:

\begin{itemize}

	\item $\|x^2\|=\|x\|^2$ for $x\in \A$,

	\item $\|x^2\|\le \|x^2+y^2\|$ for $x,y\in\A$.

\end{itemize}

A \emph{$JB^\ast$-algebra} is a complex Jordan Banach algebra $\A$ equipped with an involution $^*$ and satisfying moreover the following two axioms:

\begin{itemize}

	\item $\|x^*\|=\|x\|$ for $x\in\A$,

	\item $\|\{xx^*x\}\|=\|x\|^3$ for $x\in\A$.

\end{itemize}
An element $x$ in a $JB^\ast$-algebra is called \emph{self-adjoint} if $x^*=x$. 
The \emph{self-adjoint part} of a $JB^\ast$-algebra is the real subalgebra consisting of all self-adjoint elements. 
The Jordan Banach $^*$-algebra associated with any \Ca{} is a $JB^\ast$-algebra. The self-adjoint part of any \Ca{} equipped with the Jordan product is a $JB$-algebra. The following Theorem explains the relationship of $JB$-algebras and $JB^\ast$-algebras. The first assertion is proved for example in \cite[Proposition 3.8.2]{joa}, the second one, that is much more complicated, was  proved by J.D.M.Wright in \cite[Theorem 2.8]{wright} for unital algebras.
The non-unital case can be proved using the procedure of adding a unit, see \cite[Theorem 3.3.9]{joa}.

\begin{thm}\label{L:JB-JB*}
\
\begin{itemize} 

	\item The self-adjoint part of any $JB^\ast$-algebra is a $JB$-algebra.

	\item Any $JB$-algebra is isomorphically isometric to the self-adjoint part of a unique $JB^\ast$-algebra.

\end{itemize}

\end{thm}

If $\A$ is a $JB^\ast$-algebra and $x\in\A$ is a self-adjoint element, the closed Jordan subalgebra $C(x)$ generated by $x$ is associative (this follows from \cite[Lemma 2.4.5]{joa}) and hence it is a commutative \Ca{} (this easily follows from the axioms). Therefore, a continuous functional calculus makes sense. An  element of a $JB^*$-algebra (resp. $JB$-algebra)  is  \emph{positive} if it is of the form $x^2$, where $x$ is a self-adjoint element. 
The cone of positive elements induces a partial  order  on a JB algebra (resp. self-adjoint part of a $JB^*$-algebra) in a natural way: $x\le y$ if $y-x$ is positive. 

Further, a \emph{\JBW{}} is a $JB$-algebra which is linearly isometric to the dual of a (real) Banach space, and similarly, a \emph{\JBWs{}} is a $JB^\ast$-algebra which is linearly isometric to the dual of a (complex) Banach space. 

\JBW{}s are thoroughly studied in \cite[Chapter 4]{joa}. The definition used there is different -- it is said that a $JB$-algebra $\m$ is a \JBW{} if it is \emph{monotonically complete}, i.e., if any bounded increasing net in $\m$ admits a least upper bound in $\m$, and admits a separating set of normal functionals.  A bounded linear functional $x^*$ on $\m$ is called \emph{normal} if $x^*(x_\alpha)\to x^*(x)$ for each increasing net $(x_\alpha)$ with supremum $x$. However, it is proved in \cite[Theorem 4.4.16]{joa} that a $JB$-algebra is monotonically complete and has separating set of normal functionals if and only if it is isometric to a dual space. Hence, the two definitions coincide. Moreover, the predual is unique and is formed by the normal functionals. Moreover, any \JBW{} is unital by \cite[Lemma 4.1.17]{joa}.

Unital \JBWs{}s were introduced and studied in \cite{edwards}. However, the assumption that the algebra has a unit is not restrictive, since any \JBWs{} is unital. Indeed, it was proved e.g. by Youngson in \cite[Corollary 10]{youngson} that a $JB^\ast$-algebra has a unit exactly when its closed unit ball has an extreme point. Therefore any dual $JB^\ast$-algebra has a unit because its unit ball is weak$^*$-compact and so it admits an extreme point.

The relationship of \JBW{}s and \JBWs{}s is described in the following lemma.
First we recall some definitions. A functional $\varphi$ on a $JB^*$-algebra $A$ is called  self-adjoint if $\varphi(x)=\overline{\varphi(x^*)}$ for all $x\in A$. In other words, a functional is self-adjoint if it takes real values on self-adjoint elements. A functional on a 
$JB^*$-algebra or a JBW algebra is called \emph{positive} if it takes positive values on positive elements.  A \emph{state} is a positive norm one functional. 

In the rest of this section $\m$ will denote a fixed \JBWs{}, $\m_*$ its predual, $\saM$ the self-adjoint part of $\m$
(which is a \JBW) and $\psaM$ the self-adjoint part of $\m_*$ (which is identified with the predual of $\saM$ by the following Lemma~\ref{L:JBW-JBW*}). Further, $\m_+$ will denote the positive cone of $\m$. The following lemma is essentially well known to experts in Jordan Banach algebras and it can be derived from the results of \cite{edwards}. But we have not found anywhere explicit formulation and proof of the assertions (ii) and (iii) which are very useful to easily transfer results on \JBWs{}s to \JBW{}s and vice versa. That's why we give a proof. 

\begin{lem}\label{L:JBW-JBW*}

Let $\m$ be a \JBWs{} and $\m_*$ be its predual. Moreover, let $\saM$ denote the self-adjoint part of $\m$ and $\psaM$ denote the self-adjoint part of $\m_*$. Then the following assertions hold.

\begin{itemize}

	\item[(i)] $\saM$ is weak$^*$-closed in $\m$ and hence it is a \JBW.

  \item[(ii)] The operator $\phi:\psaM\to(\saM)_*$ defined by $\phi(\omega)=\omega|_{\saM}$ is an onto linear isometry of real Banach spaces.

  \item[(iii)] The operator $\psi:\saM\times\saM\to \m$ defined by $\psi(x,y)=x+iy$ is an onto real-linear weak$^*$-to-weak$^*$ homeomorphism.

\end{itemize}

\end{lem}

\begin{proof}

(i) It is proved in  \cite[Lemma 3.1]{edwards} that $\saM$ is weak$^*$-closed and then it is deduced in  \cite[Theorem 3.2]{edwards} that $\saM$ is a \JBW.

The assertions (ii) and (iii) essentially follow from the proof of \cite[Theorem 3.2]{edwards} using the general duality theory of Banach spaces. Indeed, if $X$ is a complex Banach space, denote by $X_R$ its real version (i.e., the same space considered as a real space). Then the operator $\theta_1: \m=(\m_*)^*\to((\m_*)_R)^*$ defined by
$$\theta_1(x)(\omega)=\Re x(\omega),\quad x\in \m, \omega\in\m_*,$$
is a real-linear isometry and weak$^*$-to-weak$^*$ homeomorphism. Hence, in particular, the dual of $(\m_*)_R$ is canonically isometric to $\m_R$. Since $\saM$ is weak$^*$ closed in $\m$ (by the assertion (i)) and hence also in $\m_R$, the predual of $\saM$ is the canonical quotient of $(\m_*)_R$ by $(\saM)_\perp$. Denote the canonical quotient mapping by $\theta_2$. 
Then $\theta_2$ can be expressed by the formula
$$\theta_2(\omega)(x)=\Re \omega(x),\quad \omega\in\m_*, x\in\saM,$$
hence the operator $\phi$ defined in the assertion (ii) is the restriction of $\theta_2$ to $\psaM$. It follows that $\phi$ is a linear isomorphism of real Banach spaces. Finally, it is an isometry due to \cite[Lemma 2.1]{edwards}. 
This completes the proof of the assertion (ii).

(iii) It is clear that $\psi$ is a real-linear bijection. To see that it is  weak$^*$-to-weak$^*$ continuous, it is enough to observe that for any $\omega\in\m_*$ and $x,y\in\saM$ we have
$$\omega(\psi(x,y))=\omega(x)+i\omega(y)=\Re\omega(x)+i\Im\omega(x)+i\Re\omega(y)-\Im\omega(y)$$ 
and that $\Re\omega,\Im\omega\in (\saM)_*$.

To see that the inverse of $\psi$ is weak$^*$-to-weak$^*$ continuous as well observe first that
$$\psi^{-1}(a)=\left(\frac{a+a^*}2,\frac{a-a^*}{2i}\right).$$ 
For any $\omega\in\psaM$ and $a\in\m$ we have
$$\omega\left(\frac{a+a^*}2\right)=\frac{\omega(a)+\overline{\omega(a)}}2$$
and
$$\omega\left(\frac{a-a^*}{2i}\right)=\frac{\omega(a)-\overline{\omega(a)}}{2i},$$
which proves the required continuity condition.
 \end{proof}

\section{Projections in \JBWs{}s}\label{S:proj}

The aim of this section is to prove Propositions~\ref{P:skel1} and~\ref{P:skel2}. They form one of the key steps to prove the main theorem. Proposition~\ref{P:skel1} together with Theorem~\ref{T:lepeni} implies that the predual of any \JBWs{} (or \JBW)
admits a $1$-projectional skeleton. Proposition~\ref{P:skel2} is a refinement of Proposition~\ref{P:skel1} and will enable us to construct a commutative $1$-projectional skeleton. A key tool in these results is (similarly as in \cite{BHK}) the notion of projection. Let us recall basic definitions.

An element $p$ of a \JBWs{} is said to be a \emph{projection} if $p^*=p$ and $p\circ p=p$. Similarly,
an element $p$ of a \JBW{} is called projection if $p\circ p=p$. In view of Lemma~\ref{L:JBW-JBW*} these two notions are compatible. I.e., if $\m$ is a \JBWs, then $p\in\m$ is a projection if and only if $p\in\saM$ and $p$ is a projection in the \JBW{} $\saM$. Hence, for projections in \JBWs{}s we may use the results from \cite[Section 4.2]{joa} on projections in \JBW{}s.
On the set of all the projections we consider the order inherited from $\saM$. In this order the projections form a complete lattice by \cite[Lemma 4.2.8]{joa}. Further, projections $p,q$ are called \emph{orthogonal} if $p\circ q=0$.

For a projection $p\in\m$ we define the operator $U_p$ on $\m$ by the formula
$$U_p(x)=(\{pxp\}=)2 p\circ(p\circ x)-p\circ x, \quad x\in\m.$$
The following lemma summarizes basic properties of the operator $U_p$. Most of them are known to experts, but we indicate the proof for the sake of completeness.

\begin{lem}\label{L1-Up}

Let $p\in\m$ be a projection. Then the following assertions are valid.
\begin{itemize}

	\item[(i)] $U_p$ is a weak$^*$-to-weak$^*$ continuous linear projection of norm one.

	\item[(ii)] $\saM$ is invariant for $U_p$.

	\item[(iii)] $U_p[\m]$ is a $JBW^*$-subalgebra of $\m$.

	\item[(iv)] If $x\in U_p[\m]\cap\saM$ and $y\in\saM$ is such that $0\le y\le x$, then $y\in U_p[\m]$.

	\item[(v)] Let $x\in\m$. Then $x\in U_p[\m]$ if and only if $p\circ x=x$.

	\item[(vi)] If $q\in \m$ is a projection such that $q\le p$, then $U_pU_q=U_qU_p=U_q$.

  \item[(vii)] $U_p^*[\m_*]\subset\m_*$.

  \item[(viii)] The positive cone of $\m$ is invariant for $U_p$ and the positive cone of $\m_*$ is invariant for $U_p^*$. 

  \item[(ix)] If $q\in\m$ is a projection, then $q\le p$ if and only if $U_q^*[\m_*]\subset U_p^*[\m_*]$.

  \item[(x)] If $q,r\in\m$ are projections such that $p,q,r$ are pairwise orthogonal, then $U_{p+q}U_{p+r}=U_p$.

\end{itemize}

\end{lem}

\begin{proof}
It is clear that $U_p$ is a linear operator and that $U_p(x^*)=U_p(x)^*$ for $x\in\m$, in particular $\saM$ is invariant of $U_p$. Hence the assertion (ii) is proved. $U_p$ is a projection by \cite[(2.61) on p. 46]{joa}.
The weak$^*$-to-weak$^*$ continuity of $U_p$ on $\saM$ follows from \cite[Corollary 4.1.6]{joa}, the weak$^*$-to-weak$^*$ continuity on $\m$ then follows from Lemma~\ref{L:JBW-JBW*}(iii) using the already proved assertion (ii).
Hence, the assertion (vii) follows. 
To complete the proof of the assertion (i) it remains to show that $\|U_p\|\le1$.
Since $\|p\|=1$ and $U_p(x)=\{pxp\}$ for each $x\in \m$, the estimate follows from the inequality $\|\{xyz\}\|\le \|x\|\cdot\|y\|\cdot\|z\|$ (see \cite[Proposition 3.4.17]{nonasoc}).

The `if' part of (v) is obvious, the `only if' part follows from \cite[(2.62) on p. 46]{joa}.
Further, by \cite[Lemma 4.1.13]{joa} we get the assertion (iv) and that $U_p[\saM]$ is a $JBW$-subalgebra of $\saM$.
It follows that $U_p[\m]=U_p[\saM]+i U_p[\saM]$ is a $JBW^*$-subalgebra of $\m$ (using Lemma~\ref{L:JBW-JBW*}(iii)),
which proves the assertion (iii).
The assertion (vi) is proved in \cite[Proposition 2.26]{AS}.
The positive cone of $\m$ is invariant for $U_p$ by \cite[Proposition 3.3.6]{joa} applied to the algebra $\saM$.
The invariance of the positive cone of $\m_*$ then easily follows and (viii) is proved. 

(ix) If $q\le p$, then it follows from (vi) that $U_q^*=U_p^*U_q^*$, hence $U_q^*[\m_*]\subset U_p^*[\m_*]$.
Conversely, suppose that $U_q^*[\m_*]\subset U_p^*[\m_*]$. Then clearly $U_p^*|_{\m_*} U_q^*|_{\m_*}=U_q^*|_{\m_*}$, hence $U_qU_p=U_q$. It follows that
$$U_q(1-p)=U_q U_p(1-p)=U_q(p-p)=0,$$
hence $q\le p$ by \cite[Lemma 4.2.2(iv)$\Rightarrow$(iii)]{joa}.

(x) First observe that whenever $q,r$ are mutually orthogonal projections, then $q\circ U_r(x)=0$ for each $x\in\m$.
Indeed, $r+q$ is a projection and $r+q\ge r$, hence
$$q\circ U_r(x)= (q+r)\circ U_r(x) - r \circ U_r(x)= U_r(x)-U_r(x)=0$$
by (vi) and (v). It follows that
$$\begin{aligned}
U_{p+q}(U_{p+r}(x)) 
&= 2 (p+q)\circ( (p+q)\circ U_{p+r}(x))- (p+q)\circ U_{p+r}(x)
\\&= 2 (p+q)\circ( p\circ U_{p+r}(x))- p\circ U_{p+r}(x)
\\&= U_p(U_{p+r}(x))+ 2 q\circ (p\circ U_{p+r}(x))
= U_p(x).
\end{aligned}$$
Indeed, the first equality is just a definition of $U_{p+q}$, the second follows from mutual orthogonality of $q$ and $p+r$,
the third one follows from the definition of $U_p$. Finally, to show the fourth equality it is enough to observe that $p\circ U_{p+r}(x)\in U_{p+r}[\m]$ by (iii).
\end{proof}

The following lemma follows from \cite[Corollary 2.6]{rec3}, even though it is not completely obvious for non-experts.
Since the proof in \cite{rec3} uses advanced structural results on Jordan algebras, we give a more direct and elementary proof.

\begin{lem}\label{L:directed}
Let $(p_n)$ be an increasing sequence of projections in $\m$ with supremum $p$. 
Then for each $\varrho\in\m_*$ we have $U_{p_n}^*\varrho\to U_p^*\varrho$ in norm.
\end{lem}

\begin{proof} Since any $\varrho\in\m_*$ is a linear combination of four normal states
(this follows from Lemma~\ref{L:JBW-JBW*} and \cite[Proposition 4.5.3]{joa}), it is enough to prove the convergence 
  in case $\varrho$ is a normal state. Hence assume that $\varrho$ is a normal state. If $q\in\m$ is any projection,
 the Cauchy-Schwarz inequality yields
   \begin{equation}\label{eq:CS} |\varrho(q\circ x)|^2\le\varrho(q\circ q^*)\varrho(x^*\circ x)\le \varrho(q) \|x\|^2,\quad x\in\m.\end{equation} 
  Fix $n\in\N$. Observe that $p$ and $p_n$ operator commute, i.e., 
  $$p\circ(p_n\circ x)=p_n\circ (p\circ x),\quad \mbox{ for }x\in\m.$$
  Indeed, this follows from \cite[Lemma 2.5.5(ii)$\Rightarrow$(i)]{joa} as $p_n\in U_p[\m]$ due to Lemma~\ref{L1-Up}(iv) and hence
  $p\circ p_n = U_p(p_n)=p_n$ due  to Lemma~\ref{L1-Up}. Therefore we have for each $x\in\m$
$$\begin{aligned} U_p(x)&-U_{p_n}(x)= 2 p \circ (p \circ x) - p\circ x - 2 p_n \circ (p_n \circ x) + p_n\circ x 
 \\&
 = 2[ p \circ (p \circ x)- p_n \circ (p \circ x) + p \circ (p_n \circ x) -p_n\circ (p_n\circ x)]
 -p\circ x + p_n\circ x\,
 \\&=2 (p-p_n)\circ ((p+p_n)\circ x)-(p-p_n)\circ x.\end{aligned}$$
Hence, combining this with \eqref{eq:CS} we get
$$\begin{aligned} | U_p^*\varrho(x)-  U_{p_n}^*\varrho(x) |&= |\varrho(U_p x-U_{p_n}x)|
\\&\le 2\varrho(p-p_n)^{1/2}\cdot\|(p+p_n)\circ x\| + \varrho(p-p_n)^{1/2}\cdot\|x\|
\\&\le 5\varrho(p-p_n)^{1/2}\|x\|,\end{aligned}$$
therefore
$$\| U_p^*\varrho-  U_{p_n}^*\varrho\|\le 5 \varrho(p-p_n)^{1/2}.$$
Since $\varrho(p-p_n)\to 0$ by normality, we conclude that $U_{p_n}^*\varrho\to U_p^*\varrho$ in norm
and the proof is completed.
\end{proof}

\begin{lem}\label{L:nosic} Let $\omega\in\m_*$ be a normal state.

\begin{itemize}

	\item[(i)] There exists the smallest projection in $\m$ such that $\omega(p_\omega)=1$. (It is called the \emph{support of $\omega$}.)

	\item[(ii)] $\omega(x)=\omega(p_\omega\circ x)$ for each $x\in\m$.

	\item[(iii)] Let $x\in\m_+$ be such that $\omega(x)=0$. Then $p_\omega\circ x=0$.  

\end{itemize}

\end{lem}

\begin{proof} The assertion (i) is proved in \cite[Lemma 5.1]{AS}. 

Let us prove the assertion (ii). For each $x\in\m$ the Cauchy-Schwarz inequality yields
$$|\omega((1-p_\omega)\circ x)|^2\le \omega((1-p_\omega)\circ(1-p_\omega))\cdot\omega(x^*\circ x)= \omega(1-p_\omega)\cdot\omega(x^*\circ x)=0,$$
hence $\omega(x)=\omega(p_\omega\circ x)$. 

To prove (iii) suppose that $x\in\m_+$ and $\omega(x)=0$. Denote by $r(x)$ the range projection of $x$ (i.e., the smallest projection satisfying $r(x)\circ x=x$, see \cite[Lemma 4.2.6]{joa}). Then $\omega(r(x))=0$ by \cite[Proposition 2.15]{AS}.
Hence $\omega(1-r(x))=1$, so $1-r(x)\ge p_\omega$. It follows that $r(x)\circ p_\omega=0$. Since $r(p_\omega)=p_\omega$,
\cite[Proposition 2.16]{AS} shows that $x\circ p_\omega=0$ and the proof is completed.
\end{proof}

A projection $p\in\m$ is said to be \emph{$\sigma$-finite} if any orthogonal system of smaller projections is countable.
The following lemma characterizes $\sigma$-finite projections. A similar result in a different setting is given in \cite[Theorem 3.2]{ER}.

\begin{lem}\label{L:sigmaf} Let $p\in\m$ be a nonzero projection. Then $p$ is $\sigma$-finite if and only if $p=p_\omega$ for a normal state $\omega\in\m_*$.

\end{lem}

\begin{proof} Suppose first that $p=p_\omega$ for a normal state $\omega$. Let $q\le p$ be any nonzero projection. Then $\omega(q)>0$ since otherwise $p-q$ would be a projection strictly smaller than $p$ with $\omega(p-q)=1$. By a standard argument we obtain that $p$ is $\sigma$-finite. 

To prove the converse observe first that for any nonzero projection $p$ there is a normal state $\omega$ with $p_\omega\le p$. Indeed, let $\omega_0$ be a normal state such that $\omega_0(p)>0$. Set $\omega=\frac1{\omega_0(p)}U_p^*(\omega_0)$. Then $\omega$ is a positive functional by Lemma~\ref{L1-Up}(viii). Moreover, $\omega(1)=\omega(p)=1$, hence $\omega$ is a normal state and $p_\omega\le p$. 

Now, given any $\sigma$-finite projection $p$, by the previous paragraph and Zorn lemma we get a sequence of normal states $(\omega_n)$ such that their supports $p_{\omega_n}$ are pairwise orthogonal and their sum is $p$. Let $\omega=\sum_{n=1}^\infty 2^{-n}\omega_n$. Then $\omega$ is a normal state and $\omega(p)=1$. Moreover, $p=p_\omega$ as
$\omega(q)>0$ for each nonzero projection $q\le p$. (Indeed, suppose that $\omega(q)=0$. It follows that $\omega(1-q)=1$, hence $1-q\ge p_{\omega_n}$. It follows that $1-q\ge p$, hence $q\le 1-p$.)
\end{proof}

The following lemma establishes $\sigma$-completeness of the lattice of $\sigma$-finite projections.  A similar result in a different setting is given in \cite[Theorem 3.4]{ER}.

\begin{lem}\label{L:spojitost} Let $(p_n)$ be a sequence of $\sigma$-finite projections. Then its supremum is $\sigma$-finite as well.
\end{lem}

\begin{proof} Denote by $p$ the supremum of the sequence $(p_n)$.
By Lemma~\ref{L:sigmaf} there is a sequence of normal states $(\omega_n)$ such that $p_n=p_{\omega_n}$.
Let $\omega=\sum_{n=1}^\infty 2^{-n}\omega_n$. Then $\omega$ is a normal state. Moreover, since 
$$0\le\omega_n(1-p)\le\omega_n(1-p_n)=0,$$
for each $n\in\N$, we get $\omega(1-p)=0$ and hence $\omega(p)=1$, so $p_\omega\le p$. Set $q=p-p_\omega$.
Then $\omega(q)=0$, hence $\omega_n(q)=0$ for each $n$. Therefore we have for each $n\in\N$ $1-q\ge p_n$, hence $1-q\ge p$, so
$$p=p\circ(1-q)=p-p\circ q=p-p+p\circ p_\omega=p\circ p_\omega=p_\omega.$$
Hence $p$ is $\sigma$-finite by Lemma~\ref{L:sigmaf}.
\end{proof}

\begin{lem}\label{L:cover} Let $\omega\in\m_*$ be arbitrary. Then there is a $\sigma$-finite projection $p\in\m$ such that 
$\omega=U_p^*(\omega)$.
\end{lem}

\begin{proof} Let $\omega\in\m_*$ be arbitrary. Then there are four normal states $\omega_1,\dots,\omega_4$ and numbers $\alpha_1,\dots,\alpha_4\ge 0$ such that
$$\omega=\alpha_1\omega_1-\alpha_2\omega_2+i(\alpha_3\omega_3-\alpha_4\omega_4).$$
Set
$p_j=p_{\omega_j}$ for $j=1,\dots,4$. By Lemma~\ref{L:nosic}(ii) we have for each $j=1,\dots,4$ 
$$\omega_j(x)=\omega_j(p_j\circ x), \quad x\in \m,$$
hence clearly $\omega_j=U_{p_j}^*(\omega_j)$. Let $p$ be the supremum of the projections $p_1,\dots,p_4$. By Lemmata~\ref{L:sigmaf} and~\ref{L:spojitost} the projection $p$ is $\sigma$-finite. Moreover, $\omega_j\in U_p^*[\m_*]$
by Lemma~\ref{L1-Up}(ix). Thus $\omega\in U_p^*[\m_*]$.
\end{proof}

We continue with a proposition which is an analogue of \cite[Lemma 3.3]{BHK}. Recall that a Banach space is called \emph{weakly compactly generated} (shortly \emph{WCG}) if it contains a linearly dense weakly compact subset. WCG spaces form a subclass of WLD spaces by \cite[Proposition 2]{AL}. In the proof below we use the well-known easy fact that if there is a bounded linear operator from a Hilbert space to a Banach space $X$ with dense range, then $X$ is WCG (cf. \cite[Proposition 2.2]{BHK}).

\begin{pr}\label{L:wcg} Let $p\in\m$ be a $\sigma$-finite projection. Then $U_p^*[\m_*]$ is WCG. \end{pr}

\begin{proof} If $p=0$ the assertion is trivial. Suppose that $p\ne 0$ and let $\omega$ be a normal state such that $p=p_\omega$ provided by Lemma~\ref{L:sigmaf}. Let us define an operator $\Phi:\m\to \m_*$ by the following formula:
$$\Phi(a)(x)=\omega(a\circ U_p(x)), \quad a,x\in\m,$$
i.e., $\Phi(a)=U_p^*T_a^*\omega$, where the operator $T_a$ is defined by $x\mapsto a\circ x$. Since $T_a$ is weak$^*$-to-weak$^*$ continuous by \cite[Corollary 4.1.6]{joa}, it is clear that $\Phi$ is a linear operator mapping $\m$ into $\m_*$, in fact into $U_p^*[\m_*]$.

Let us further prove that the range of $\Phi$ is dense in $U_p^*[\m_*]$. We will use Hahn-Banach theorem. To do that, suppose that $x\in\m$ is such that $\Phi(a)(x)=0$ for each $a\in M$. Take $a=(U_p(x))^*=U_p(x^*)$. Then
$$0=\Phi(U_p(x^*))(x)=\omega(U_p(x^*)\circ U_p(x)).$$
As $U_p(x^*)\circ U_p(x)$ is positive, we obtain by Lemma~\ref{L:nosic} (iii) that $p\circ(U_p(x^*)\circ U_p(x))=0$, hence $U_p(x^*)\circ U_p(x)=0$ by Lemma~\ref{L1-Up}(iii),(v).
It follows that $U_p(x)=0$, hence $\varrho(x)=0$ for each $\varrho\in U_p^*[\m_*]$. Hence, the Hahn-Banach theorem yields the
density of the range of $\Phi$ in $U_p^*[\m_*]$.

Finally, we have by the Cauchy-Schwarz inequality
$$|\Phi(a)(x)|^2=|\omega(a\circ U_p(x))|^2\le \omega(a\circ a^*)\omega(U_p(x^*)\circ U_p(x))\le \omega(a\circ a^*)\|x\|^2,$$
hence $\|\Phi(a)\|\le \omega(a\circ a^*)^{1/2}$ for each $a\in \m$. Define $H_\omega$ to be the Hilbert space made
by the standard procedure of factorization and completion from $\m$ equipped with the semi-inner product $(x,y)\mapsto \omega(y^*\circ x)$. Then $\Phi$ induces a bounded linear map of $H_\omega$ into $U_p^*[\m_*]$ having dense range. 
This shows that $U_p^*[\m_*]$ is WCG.
\end{proof}

\begin{pr}\label{P:skel1} Let $\m$ be a \JBWs. Denote by $\Lambda$ the set of all nonzero $\sigma$-finite projections in $\m$ 
equipped with the standard order. For $p\in\Lambda$ let $Q_p$ denote the restriction of $U_p^*$ to $\m_*$.
Then $\Lambda$ is a directed set and the following conditions are fulfilled.
\begin{itemize}

	\item[(i)] $Q_p$ is a linear projection, $\|Q_p\|=1$ for each $p\in\Lambda$.

	\item[(ii)] $Q_p[\m_*]$ is WCG for each $p\in\Lambda$.

	\item[(iii)] If $p_1,p_2\in\Lambda$ are such that $p_1\le p_2$, then $Q_{p_1}Q_{p_2}=Q_{p_2}Q_{p_1}=Q_{p_1}$.

	\item[(iv)] If $p_1\le p_2\le \dots$ are in $\Lambda$, then $p=\sup_n p_n$ exists in $\Lambda$ and, moreover
	$Q_{p_n}\to Q_p$ is the strong operator topology, in particular $Q_p[\m_*]=\overline{\bigcup_n Q_{p_n}[\m_*]}$

	\item[(v)] $\m_*=\bigcup_{p\in\Lambda} Q_p[\m_*]$.

\end{itemize}
Moreover, $Q_p[\psaM]\subset\psaM$ and
$$\begin{aligned}\bigcup_{p\in\Lambda} Q_p^*[(\m_*)^*]&=\bigcup_{p\in\Lambda} U_p[\m]\\& =\{x\in \m:\exists p\in \m\mbox{ a $\sigma$-finite projection}: U_p(x)=x\}\\
&=\{x\in \m:\exists p\in \m\mbox{ a $\sigma$-finite projection}: p\circ x=x\},\\
\bigcup_{p\in\Lambda} Q_p^*[(\psaM)^*]&=\bigcup_{p\in\Lambda} U_p[\saM]\\& =\{x\in \saM:\exists p\in \m\mbox{ a $\sigma$-finite projection}: U_p(x)=x\}\\
&=\{x\in \saM:\exists p\in \m\mbox{ a $\sigma$-finite projection}: p\circ x=x\}.\end{aligned}$$
Further, $Q_{p_1}[\m_*]\subset Q_{p_2}[\m_*]$ if and only if $p_1\le p_2$.
\end{pr}

\begin{proof}
$\Lambda$ is directed by Lemma~\ref{L:spojitost}. 
The assertion (i) follows from Lemma~\ref{L1-Up}(i),(vii); the assertion (ii) is proved in Proposition~\ref{L:wcg};
(iii) follows from Lemma~\ref{L1-Up}(vi), the assertion (iv) follows by using Lemma~\ref{L:directed} and Lemma~\ref{L:spojitost}
and the assertion (v) is proved in Lemma~\ref{L:cover}. The invariance of $\psaM$ follows by Lemma~\ref{L1-Up}(ii) using Lemma~\ref{L:JBW-JBW*}. The formulas follow from the fact that $Q_p^*=U_p$ for each $p\in\Lambda$ and from Lemma~\ref{L1-Up}(v).
The final equivalence is due to Lemma~\ref{L1-Up}(ix).
\end{proof}

\begin{pr}\label{P:skel2} Let $\m$ be a \JBWs. Then there is an orthogonal family of nonzero $\sigma$-finite projections $(p_\alpha)_{\alpha\in\Gamma}$ with sum equal to $1$. Denote by $\Lambda_0$ the family of all the nonempty countable subsets of $\Gamma$ ordered by inclusion. For any $C\in\Lambda_0$ define $p_C=\sum_{\alpha\in C} p_\alpha$ and define $R_C=Q_{p_C}$.

Then the system $R_C$, $C\in\Lambda_0$, enjoys all the properties of the system $Q_p$, $p\in\Lambda$, from Proposition~\ref{P:skel1}. Moreover, it is commutative, i.e., $R_{C_1}R_{C_2}=R_{C_2}R_{C_1}$; and
$$\bigcup_{C\in\Lambda_0} R_C^*[(\m_*)^*]=\bigcup_{p\in\Lambda} Q_p^*[(\m_*)^*]=\bigcup_{p\in\Lambda} U_p[\m].$$
\end{pr}

\begin{proof} Similarly as in the proof of Lemma~\ref{L:sigmaf} we see that for any nonzero projection $p\in\m$ there is a nonzero $\sigma$-finite projection $q\le p$. Therefore the existence of the system $(p_\alpha)_{\alpha\in\Gamma}$ follows from the Zorn lemma.
Further, it is clear that $\Lambda_0$ is directed. The projections $p_C$, $C\in\Lambda_0$ are $\sigma$-finite by Lemma~\ref{L:spojitost}.
Hence the analogues of assertions (i)--(iv) from Proposition~\ref{P:skel1} are obviously fulfilled, as well as the final equivalence. To prove the analogue of (v) and the equality it is enough to show that for any $p\in \Lambda$ there is $C\in\Lambda_0$ such that $p\le p_C$. So fix $p\in\Lambda$. Lemma~\ref{L:sigmaf} yields a normal state $\omega\in\m_*$ with 
$p=p_\omega$. Then it follows by normality of $\omega$ that
$$1=\omega(1)=\sum_{\alpha\in\Gamma}\omega(p_\alpha).$$
Let $C=\{\alpha\in\Gamma:\omega(p_\alpha)>0\}$. Then $C$ is countable, hence $C\in\Lambda_0$. Moreover, $\omega(p_C)=1$, hence
$p_C\ge p_\omega=p$.
Finally, to show the commutativity observe that Lemma~\ref{L1-Up}(x) implies $R_{C_1}R_{C_2}=R_{C_1\cap C_2}$ for any $C_1,C_2\in\Lambda_0$ (and $R_{C_1}R_{C_2}=0$ if $C_1\cap C_2=\emptyset$).
\end{proof}

\section{Method of elementary submodels}\label{S:elem}

In this section we briefly recall some basic facts concerning the method of elementary models which will be used to prove Theorem~\ref{T:lepeni} and the main theorem. This set-theoretical method can be used in various branches of mathematics. 
The use in topology was illustrated by A.~Dow in \cite{dow}, in functional analysis it was used by P.~Koszmider in \cite{kos05}. This method was later used by W. Kubi\'s in \cite{kubisSkeleton} to construct projectional skeletons in certain  Banach spaces. In \cite{cuth-fm} the method has been slightly simplified and specified, and it was used to proving separable reduction theorems. We briefly recall some basic facts (more details and explanations may be found e.g. in \cite{cuth-fm} and \cite{cuka-cejm}). We use the approach of \cite{cuth-fm}.

We start by recalling some definitions. Let $N$ be a fixed set and $\phi$ a formula in the language of the set theory. Then the {\em relativization of $\phi$ to $N$} is the formula $\phi^N$ which is obtained from $\phi$ by replacing each quantifier of the form ``$\forall x$'' by ``$\forall x\in N$'' and each quantifier of the form ``$\exists x$'' by ``$\exists x\in N$''.

If $\phi(x_1,\ldots,x_n)$ is a formula with all free variables displayed (i.e., a formula whose free variables are exactly $x_1,\ldots,x_n$) then  $\phi$ is said to be {\em absolute for $N$} if
\[
\forall a_1,\ldots,a_n\in N\quad (\phi^N(a_1,\ldots,a_n) \leftrightarrow \phi(a_1,\ldots,a_n)).
\]
A list of formulas, $\phi_1,\ldots,\phi_n$, is said to be {\em subformula closed} if  every subformula of a formula in the list is also contained in the list.

The method is mainly based on the following theorem (a proof can be found in \cite[Chapter IV, Theorem 7.8]{Kunen}).

\begin{thm}\label{T:countable-model}
Let $\phi_1, \ldots, \phi_n$ be any formulas and $Y$ any set. Then there exists a set $M \supset Y$ such that
$\phi_1, \ldots, \phi_n \text{ are absolute for } M$ and $|M| \leq \max(\aleph_0,|Y|)$.
\end{thm}

To be able to use Theorem~\ref{T:countable-model} effectively, we will use the following notation.

Let $\phi_1, \ldots, \phi_n$ be any formulas and $Y$ be any countable set.
Let $M \supset Y$ be a countable set such that $\phi_1, \ldots, \phi_n$ are absolute for $M$.
Then we say that $M$ is an \emph{elementary model for $\phi_1,\ldots,\phi_n$ containing $Y$}.
This is denoted by $M \prec (\phi_1,\ldots,\phi_n; Y)$.

The fact that certain formula is absolute for $M$ will always be used in order to satisfy the assumption of the following lemma from \cite[Lemma 2.3]{cuthRmoutilZeleny}. Using this lemma we can force the model $M$ to contain all the needed objects created (uniquely) from elements of $M$.

\begin{lem}\label{l:unique-M}
Let $\phi(y,x_1,\ldots,x_n)$ be a formula with all free variables shown and $Y$ be a countable set.
Let $M$ be a fixed set, $M \prec (\phi, \exists y \colon \phi(y,x_1,\ldots,x_n);\; Y)$, and
$a_1,\ldots,a_n \in M$ be such that there exists a set $u$ satisfying
$\phi(u,a_1,\ldots,a_n)$. Then there exists $u \in M$ such that $\phi(u,a_1,\ldots,a_n)$.
\end{lem}

\begin{proof}Let us give here the proof just for the sake of completeness. Using the absoluteness of the formula $\exists u\colon \phi(u,x_1,\ldots,x_n)$ there exists $u\in M$ satisfying $\phi^M(u,a_1,\ldots,a_n)$.
Using the absoluteness of $\phi$ we get, that for this $u\in M$ the formula $\phi(u,a_1,\ldots,a_n)$ holds.
\end{proof}

We shall also use the following convention.

\begin{convention}
Whenever we say ``\emph{for any suitable model $M$ (the following holds \dots)}''
we mean that  ``\emph{there exists a list of formulas $\phi_1,\ldots,\phi_n$ and a countable set $Y$ such that for every $M \prec (\phi_1,\ldots,\phi_n;Y)$ (the following holds \dots)}''.
\end{convention}

By using this new terminology we loose the information about the formulas $\phi_1,\ldots,\phi_n$ and the set $Y$.
However, this is not important in applications.

The next lemma summarizes several properties of ``sufficiently large'' elementary models.

%

\begin{lem}\label{l:predp}
There are formulas $\theta_1,\dots,\theta_m$ and a countable set $Y_0$ such that any $M\prec(\theta_1,\ldots,\theta_m;\; Y_0)$ satisfies the following conditions:
\begin{itemize}

	\item[(i)] $\rr,\cc,\Q,\Q+i\Q,\zz,\nn\in M$ and the operations of the addition and multiplication on $\cc$ and the standard order on $\rr$ belong to $M$.

	\item[(ii)] If $f\in M$ is a mapping, then $\dom(f)\in M$, $\rng(f)\in M$ and $f[M]\subset M$. Further, for any $A\in M$ we have $f[A]\in M$ as well.

	\item[(iii)] If $A$ is finite, then $A\in M$ if and only if $A\subset M$.
   
  \item[(iv)] If $x_1,\dots,x_n$ are arbitrary, then $x_1,\dots,x_n\in M$ if and only if the ordered $n$-tuple $(x_1,\dots,x_n)$ is an element of $M$.
		
	\item[(v)] If $A\in M$ is a countable set, then $A\subset M$.

	\item[(vi)] If $A,B\in M$, then $A\cup B\in M$, $A\cap B\in M$, $A\setminus B\in M$.

	\item[(vii)] If $A,B\in M$, then $A\times B\in M$.
	
	\item[(viii)] If $X\in M$ is a real vector space, then $X\cap M$ is $\Q$-linear.

	\item[(ix)] If $X\in M$ is a complex vector space, then $X\cap M$ is $(\Q+i\Q)$-linear.

	\item[(x)] If $X\in M$ is a Banach space, then $X^*\in M$ as well.

	\item[(xi)] If $X,Y$ are Banach spaces and $T:X\to Y$ is a bounded linear operator such that $X,Y,T\in M$, then  $T^*\in M$ as well.

	\item[(xii)] If $X\in M$ is a separable metric space, then there is a dense countable set $C\subset X$ with $C\in M$.

	\item[(xiii)] If $\Gamma\in M$ is an up-directed set, then $\Gamma\cap M$ is also up-directed.
\end{itemize}
\end{lem}

\begin{proof} The list $\theta_1,\dots,\theta_m$ will be formed by all the formulas provided by the results quoted in this proof, the formulas marked below by $(*)$ and their subformulas. The set $Y_0$ will contain the respective countable sets provided by the quoted results and the sets specified in (i). 

Hence, (i) is satisfied. The validity of the first three assertions of (ii) follows from \cite[Proposition 2.9]{cuth-fm}. The last property follows (using Lemma~\ref{l:unique-M}) by the absoluteness of the formula
$$\exists B\; \forall x\;(x\in B \Leftrightarrow \exists y\in A : x=f(y))\eqno{(*)}$$
and its subformulas. The assertions (iii)--(vi) follow from \cite[Proposition 2.10]{cuth-fm}. The validity of (vii) follows  
(using Lemma~\ref{l:unique-M}) by the absoluteness of the formula
$$\exists C\; \forall x (x\in C\Leftrightarrow \exists y\in A\;\exists z\in B: x=(y,z))\eqno{(*)}$$
and its subformulas.

Let us prove (viii). Let $X$ be a real vector space belonging to $M$. Recall that $X$ is not just a set, but it is a quadruple
$\langle X,\rr,+,\cdot\rangle$. By (iv) we infer that the mappings $+:X\times X\to X$ and $\cdot:\rr\times X\to X$ belong to $M$ as well. By (i) and (v) we know that $\Q\subset M$. Hence, if $x\in X\cap M$ and $\lambda\in\Q$, then $\lambda x\in X\cap M$ by (iv) and (ii). Similarly, if $x,y\in X\cap M$, then $x+y\in X\cap M$. So, $X\cap M$ is $\Q$-linear.

The proof of (ix) is analogous.

(x) Let $X=\langle X,+,\cdot,\|\cdot\|\rangle\in M$. By (iv) we know that the mappings $+$, $\cdot$ and $\|\cdot\|$ belong to $M$ as well. Hence, by absoluteness of the formula
$$\begin{gathered}\exists X^*\; \forall f (f\in X^* \Leftrightarrow f\text{ is a linear functional on }X \\ \& \exists r\in \rr 
\forall x\in X (\|x\|\le 1\Rightarrow |f(x)|\le r)) \end{gathered}\eqno{(*)}$$
and its subformulas we get (using Lemma~\ref{l:unique-M}) that $X^*\in M$ as a set. Moreover, by (vii) we get $X^*\times X^*\in M$. Since the operations $+$ and $\cdot$ on $X^*$ can be uniquely described by  suitable formulas (we mark them by $(*)$), these operations belong to $M$ as well. Similarly we can achieve that the norm on $X^*$ belongs to $M$, hence $X^*\in M$ as a normed linear space by (iv).

(xi) By (x) we get $X^*,Y^*\in M$. By the absoluteness of the formula
$$\exists T^* (T^*:Y^*\to X^* \& \forall y^*\in Y^*: T^*(y^*)=y^*\circ T)\eqno{(*)}$$
and its subformulas we get $T^*\in M$ (using Lemma~\ref{l:unique-M}).

(xii) Let $X=\langle X,d\rangle$ be a separable metric space belonging to $M$. A countable dense subset of $X$ belonging to $M$ can be obtained by the absoluteness of the formula
$$\begin{gathered}\exists D (D\subset X \& \exists f (f \mbox{ is a mapping of $\nn$ onto }D)\\ \& \forall x\in X\forall r\in \rr(r>0\Rightarrow 
\exists y\in D :d(x,y)<r))\end{gathered}\eqno{(*)}$$
and its subformulas using Lemma~\ref{l:unique-M}.

(xiii) Let $\Gamma=(\Gamma,\le)$ be an up-directed set in $M$. Take $a,b\in \Gamma\cap M$. By the absoluteness of the formula
$$\exists c\in\Gamma: a\le c\& b\le c\eqno{(*)}$$
we can (using Lemma~\ref{l:unique-M}) find such a $c$ in $\Gamma\cap M$.
\end{proof}

\section{Amalgamating projectional skeletons}\label{S:final}

The aim of this section is to prove Theorem~\ref{T:lepeni} and Theorem~\ref{T:main}. It will be done using the method of elementary submodels described in the previous section. We will use some ideas and results from \cite{kubisSkeleton}.
Since our setting is a bit different (due to the fact that we use the more precise approach of \cite{cuth-fm}) and that we need more precise and stronger versions of the results, we indicate also the proofs.

The first lemma is a variant of \cite[Lemma 4]{kubisSkeleton} and shows the method of constructing projections using elementary submodels.

\begin{lem}\label{le:projekce}
For a suitable elementary model $M$ the following holds: Let $X$ be a Banach space and $D\subset X^*$ an $r$-norming subspace.
If $X\in M$ and $D\in M$, then the following hold:

\begin{itemize}
	\item $\overline{X\cap M}$ is a closed linear subspace of $X$;

	\item $\overline{X\cap M}\cap (D\cap M)_\perp=\{0\}$;

	\item the canonical projection of  $\overline{X\cap M}+(D\cap M)_\perp$ onto  $\overline{X\cap M}$ along $(D\cap M)_\perp$

	has norm at most $r$.

\end{itemize}
\end{lem}

\begin{proof} Let $\phi_1,\dots,\phi_N$ be a subformula-closed list of formulas which contains the formulas from Lemma~\ref{l:predp} and the formulas below marked by $(*)$, let $Y$ be a countable subset containing the set $Y_0$ from Lemma~\ref{l:predp}. Fix an arbitrary $M\prec(\phi_1,\dots,\phi_N;Y)$.

Suppose that $X\in M$ and $D\in M$. By Lemma~\ref{l:predp}(viii,ix) $\overline{X\cap M}$ is a closed linear subspace of $X$. Therefore to prove the lemma it is enough to show that $\|x\|\le r\|x+y\|$ for any $x\in X\cap M$ and $y\in (D\cap M)_\perp$.
So, fix such $x$ and $y$. Further, let $q\in(r,\infty)\cap\Q$ be arbitrary. Since $D$ is $r$-norming,
$$\exists x^*\in D: \|x^*\|=1 \& |x^*(x)|\ge \frac1q\|x\|. \eqno{(*)}$$
Since $\frac1q\in M$ (by Lemma~\ref{l:predp}(i,iv)) we can use Lemma~\ref{l:unique-M} to find such an $x^*$ in $M$. Then
$$\|x\|\le q|x^*(x)|=q|x^*(x+y)|\le q\|x+y\|.$$
This holds for any $q\in(r,\infty)\cap \Q$, hence $\|x\|\le r\|x+y\|$ which completes the proof.
\end{proof}

The projection given by the previous lemma will be denoted by $P_M$. The important case is when $P_M$ is defined on the whole space $X$. This can be used to characterize spaces with a projectional skeleton.

\begin{lem}\label{le:generovani} Let $X$ be a Banach space and $D\subset X^*$ a norming subspace. Then the following two assertions are equivalent
\begin{itemize}
	\item[(i)] $X$ admits a projectional skeleton such that $D$ is contained in the subspace induced by the skeleton.
	\item[(ii)] For any suitable elementary model $M$ 
	$$\overline{X\cap M}+(D\cap M)_\perp=X.$$ 
\end{itemize}
\end{lem}

\begin{proof} This result is essentially proved in \cite[Theorem 15]{kubisSkeleton}. Since we are using a different approach to elementary submodels we indicate a proof.

(i)$\Rightarrow$(ii) This is essentially \cite[Lemma 14]{kubisSkeleton}. It is easy to rewrite the proof to our setting.

(ii)$\Rightarrow$(i) Let us fix a list of formulas $\phi_1,\dots,\phi_n$ containing the formulas provided by the assumption of (ii) and 
the formulas provided by Lemma~\ref{le:projekce}. Let $Y$ be a countable set containing the countable set provided by the assumption and that provided by Lemma~\ref{le:projekce}. If $M$ is a corresponding elementary model, then we have the projection
$P_M$ with range $\overline{X\cap M}$ and kernel $(D\cap M)_\perp$. Moreover, if $M_1$ and $M_2$ are two such models satisfying $M_1\subset M_2$, then $P_{M_1}P_{M_2}=P_{M_2}P_{M_1}=P_{M_1}$. Indeed, obviously $\overline{X\cap M_1}\subset\overline{X\cap M_2}$ which implies $P_{M_2}P_{M_1}=P_{M_1}$. Moreover, $\ker P_{M_2}=(D\cap M_2)_\perp\subset(D\cap M_1)_\perp=\ker P_{M_1}$, hence for any $x\in X$ we have
$$P_{M_1}(x)=P_{M_1}P_{M_2}(x)+ P_{M_1}(x-P_{M_2}(x))=P_{M_1}P_{M_2}(x).$$
Further, if $M_1\subset M_2\subset M_3\subset\dots$ is an increasing sequence of corresponding elementary models, then
$M=\bigcup_n M_n$ is again such a model and clearly $P_M[X]=\overline{\bigcup_n P_{M_n}[X]}$. Therefore, the idea is to ``put together'' all the projections $P_M$ to get a projectional skeleton. One possible way is described in \cite{kubisSkeleton} but it does not match our setting. Let us describe an alternative way.

Fix a set $R$ such that the formulas $\phi_1,\dots,\phi_n$ are absolute for $R$ and $Y\cup X\cup D\subset R$. Such $R$ exists due to 
Theorem~\ref{T:countable-model}. (Note that $R$ is not countable.) Now let $\psi$ be a Skolem function for $\phi_1,\dots,\phi_n$, $Y$ and $R$ (see \cite[Lemma 2.4]{cuka-cejm}). In particular, for any countable set $A\subset R$, $\psi(A)\prec(\phi_1,\dots,\phi_n,Y)$ and $A\subset\psi(A)$. Let
$$\Lambda=\{A\subset X\cup D; A\mbox{ countable }\&\ \psi(A)\cap (X\cup D)=A\}.$$
It easily follows from \cite[Lemma 2.4]{cuka-cejm} that $\Lambda$ is up-directed and $(P_{\psi(A)})_{A\in\Lambda}$ is a projectional skeleton. Moreover, $P_{\psi(A)}^*[X^*]=\overline{D\cap\psi(A)}^{w^*}$ and these subspaces cover $D$.
\end{proof}

The previous lemma characterizes the existence of projectional skeletons, but does not test whether the skeleton
may be chosen to be commutative. Such a characterization is given in the following lemma which is an easy consequence of the previous one.

\begin{lem}\label{le:plichko}	
	 Let $X$ be a Banach space and $D\subset X^*$ a norming subspace. Then the following two assertions are equivalent
\begin{itemize}

	\item[(i)] $D$ is contained in a $\Sigma$-subspace of $X$, i.e., $X$ admits a commutative projectional skeleton such that $D$ is contained in the subspace induced by the skeleton.

	\item[(ii)] There is a list of formulas $\phi_1,\dots,\phi_n$ and a countable set $Y$ such that the following holds:

	\begin{itemize}

	\item[$\bullet$] 	$\overline{X\cap M}+(D\cap M)_\perp=X$ for any $M\prec(\phi_1,\dots,\phi_n;Y)$.

	\item[$\bullet$] $P_{M_1}$ and $P_{M_2}$ commute whenever $M_j\prec(\phi_1,\dots,\phi_n;Y)$ for $j=1,2$.

\end{itemize}

\end{itemize}

\end{lem}

\begin{proof} The implication (i)$\Rightarrow$(ii) follows by a slight refinement of the proof of the respective implication in Lemma~\ref{le:generovani}. The converse one follows immediately from the proof of (ii)$\Rightarrow$(i) of Lemma~\ref{le:generovani} since the skeleton is built from projections of the form $P_M$.
\end{proof}

Now we proceed to the proof of Theorem~\ref{T:lepeni}. It will be done using Lemma~\ref{le:generovani}. We will further
need a strengthening of the implication (i)$\Rightarrow$(ii) for WLD spaces. The strengthening consists in change of quantifiers
-- we need a finite list of formulas which works for all Banach spaces simultaneously. It is the content of the following lemma.

\begin{lem}\label{le:wld} 
For any suitable elementary model $M$ the following holds:
Let $X$ be any WLD Banach space satisfying $X\in M$. Then $X=\overline{X\cap M}+(X^*\cap M)_\perp$.
\end{lem}

\begin{proof} We essentially follow the proof \cite[Proposition 6]{kubisSkeleton} with necessary modifications.  Let $\phi_1,\dots,\phi_N$ be a subformula-closed list of formulas which contains the formulas from Lemma~\ref{l:predp}, the formulas provided by Lemma~\ref{le:projekce} and the formulas below marked by $(*)$. Let $Y$ be a countable subset containing the set $Y_0$ from Lemma~\ref{l:predp} and the set provided by Lemma~\ref{le:projekce}. Fix an arbitrary $M\prec(\phi_1,\dots,\phi_N;Y)$.

By Lemma~\ref{l:predp}(x) we have $X^*\in M$ as well. It follows from Lemma~\ref{le:projekce} that $\overline{X\cap M}+(X^*\cap M)_\perp$ is a closed subspace of $X$. Hence, if $X \ne \overline{X\cap M}+(X^*\cap M)_\perp$, 
we may find a nonzero functional $z^*\in X^*$ which is zero both on $X\cap M$ and on $(X^*\cap M)_\perp$.
Since $X$ is WLD,
$$\exists \Gamma\subset X: \overline{\sp \Gamma}=X\ \&\ \forall x^*\in X^*: \{x\in \Gamma: x^*(x)\ne 0\}\mbox{ is countable.}\eqno{(*)}$$
By elementarity we may choose such a $\Gamma$ in $M$. Since $z^*\ne0$, we can find $x\in\Gamma$ with $z^*(x)\ne0$.
Since $z^*\in((X^*\cap M)_\perp)^\perp=\overline{X^*\cap M}^{w^*}$ (by the Bipolar Theorem), there is $y^*\in X^*\cap M$ with $y^*(x)\ne 0$. On the other hand, by the absoluteness of the formula
$$\exists C: (C\subset \Gamma\ \&\ \forall y\in\Gamma:(y\in C\Leftrightarrow y^*(y)\ne 0)) \eqno{(*)}$$
we get that
$$\{y\in \Gamma : y^*(y)\ne 0\}\in M.$$
Since the set on the left-hand side is countable, by Lemma~\ref{l:predp}(v) we get that $\{y\in \Gamma : y^*(y)\ne 0\}\subset M$, in particular $x\in M$. But then $z^*(x)=0$, a contradiction completing the proof.
\end{proof}

The following lemma together with Lemma~\ref{le:generovani} yield the proof of Theorem~\ref{T:lepeni}.

\begin{lem}\label{le:lepeni} For any suitable elementary model $M$ the following holds:
Let $X$ be a Banach space and $(R_\lambda)_{\lambda\in \Lambda}$ a family of projections with the properties listed in Theorem~\ref{T:lepeni}. Denote $D= \bigcup_{\lambda\in\Lambda}R_\lambda^*[X^*]$. If $X$, $D$ and $(R_\lambda)_{\lambda\in \Lambda}$ belong to $M$, then $X=\overline{X\cap M}+(D\cap M)_\perp$.
\end{lem}

\begin{proof}  Let $\phi_1,\dots,\phi_N$ be a subformula-closed list of formulas which contains the formulas from Lemma~\ref{l:predp}, the formulas provided by Lemmata~\ref{le:projekce} and~\ref{le:wld} and the formulas below marked by $(*)$. Let $Y$ be a countable subset containing the set $Y_0$ from Lemma~\ref{l:predp} and the sets provided by Lemmata~\ref{le:projekce} and~\ref{le:wld}. Fix an arbitrary $M\prec(\phi_1,\dots,\phi_N;Y)$ such that $\{X,D,(R_\lambda)_{\lambda\in \Lambda}\}\subset M$.

Note that $X^*\in M$ due to Lemma~\ref{l:predp}(x). Since $D$ is norming (this follows easily from the properties (i) and (v) in Theorem~\ref{T:lepeni}), Lemma~\ref{le:projekce} shows that $\overline{X\cap M}+(D\cap M)_\perp$ is a closed subspace of $X$.
Hence, if $X \ne \overline{X\cap M}+(D\cap M)_\perp$, we may find a nonzero functional $z^*\in X^*$ which is zero both on $X\cap M$ and on $(D\cap M)_\perp$. Set $\Lambda_M=\Lambda\cap M$.
Since $\Lambda\in M$ by Lemma~\ref{l:predp}(ii), we infer by Lemma~\ref{l:predp}(xiii) that $\Lambda_M$ is up-directed.
Since it is countable, it follows from the properties (iii) and (iv) in Theorem~\ref{T:lepeni} that 
$\Lambda_M$ has a supremum $\lambda_0\in\Lambda$ and that $R_{\lambda_0}=SOT-\lim_{\lambda\in\Lambda_0} R_\lambda$.

Fix any $\lambda\in\Lambda_M$. Then 
$$R_\lambda[X\cap M]=R_\lambda[X]\cap M\mbox{ and }R_\lambda^*[D\cap M]=R_\lambda^*[X^*]\cap M.$$
Indeed, the inclusions $\supset$ follow from the assumption that $R_\lambda$ is a projection and the converse inclusions follow from Lemma~\ref{l:predp}. (The assertion (ii) implies that $R_\lambda\in M$, by (xi) we get $R_\lambda^*\in M$ as well, hence we can conclude by using (ii) once more.)

Since $R_\lambda[X]$ is WLD and $R_\lambda[X]\in M$ by Lemma~\ref{l:predp}(ii),  Lemma~\ref{le:wld} yields 
$$R_\lambda [X]=\overline{R_\lambda[X]\cap M}+(R_\lambda^*[X^*]\cap M)_\perp\cap R_\lambda[X].$$ Obviously $z^*$ (and so also $R_\lambda^*(z^*)$) is zero on $R_\lambda[X]\cap M$. Further, since $z^*\in ((D\cap M)_\perp)^\perp=\overline{D\cap M}^{w^*}$, we get 
$$R_\lambda^*(z^*)\in \overline{R_\lambda^*[D\cap M]}^{w^*}=\overline{R_\lambda^*[X^*]\cap M}^{w^*},$$
hence $R_\lambda^*(z^*)$ is zero on $(R_\lambda^*(X^*)\cap M)_\perp$. Thus $R_\lambda^* (z^*)=0$.
Since this holds for any $\lambda\in\Lambda_M$, we conclude $R_{\lambda_0}^* (z^*)=0$, i.e. the restriction of $z^*$ to $R_{\lambda_0}[X]$ is the zero functional.

 To complete the proof by contradiction it is enough to show that $z^*$ is zero on the kernel of $R_{\lambda_0}$ as well. 
To do that it is sufficient to prove that the kernel of $R_{\lambda_0}$ is contained in $(D\cap M)_\perp$. Hence fix $x$ in the kernel of $R_{\lambda_0}$ and $x^*\in D\cap M$. By the definition of $D$ we have
$$\exists \lambda\in\Lambda : R_\lambda^*(x^*)=x^*.\eqno{(*)}$$
By elementarity we may find such a $\lambda\in\Lambda_M$. In particular, then $\lambda\le\lambda_0$. Therefore
$$x^*(x)=R_\lambda^*(x^*)(x)=R_{\lambda_0}^*(x^*)(x)=x^*(R_{\lambda_0}(x))=x^*(0)=0.$$
This completes the proof.
\end{proof}

\begin{proof}[Proof of Theorem~\ref{T:lepeni}] Let $X$ and $(R_\lambda)_{\lambda\in\Lambda}$ be as in Theorem~\ref{T:lepeni}.
We set $D= \bigcup_{\lambda\in\Lambda}R_\lambda^*[X^*]$. By Lemma~\ref{le:lepeni} and Lemma~\ref{le:generovani} there is a projectional skeleton on $X$ such that the induced subspace of $X^*$ contains $D$. Further, it follows easily from the property (iv) that $D$ is weak$^*$-countably closed. Finally, \cite[Corollary 20]{kubisSkeleton} shows that $D$ is in fact equal to the the subspace induced by the skeleton.
\end{proof}

Now we proceed to the proof of Theorem~\ref{T:main}. To ensure commutativity of the skeleton we need some more lemmata.

\begin{lem}\label{le:kom1} For any suitable elementary model $M$ the following holds:

Let $X$ be a Banach space and $D\subset X^*$ a subspace induced by a projectional skeleton in $X$. Suppose that $X\in M$ and $D\in M$. Denote by $P_M$ the projection induced by $M$ (i.e., the projection onto $\overline{X\cap M}$ along $(D\cap M)_\perp$). Let $Q:X\to X$ be a bounded linear projection such that $Q\in M$. Then $Q$ commute with $P_M$. If $Q[X]$ is moreover separable, then $P_MQ=QP_M=Q$.

\end{lem}

\begin{proof}  Let $\phi_1,\dots,\phi_N$ be a subformula-closed list of formulas which contains the formulas from Lemma~\ref{l:predp}, the formulas provided by Lemmata~\ref{le:projekce} and~\ref{le:wld}. Let $Y$ be a countable subset containing the set $Y_0$ from Lemma~\ref{l:predp} and the sets provided by Lemmata~\ref{le:projekce} and~\ref{le:wld}. Fix an arbitrary $M\prec(\phi_1,\dots,\phi_N;Y)$.

Since $Q\in M$, by Lemma~\ref{l:predp}(ii) we have $Q[X\cap M]\subset X\cap M$, hence $P_MQP_M=Q P_M$.
Further, by Lemma~\ref{l:predp}(xi) we have $Q^*\in M$, hence $Q^*[D\cap M]\subset D\cap M$ due to Lemma~\ref{l:predp}(ii). Since $\overline{D\cap M}^{w^*}$ is the range of $P_M^*$,
we get $P_M^*Q^*P_M^*=Q^*P_M^*$, hence $P_MQP_M=P_MQ$. It follows that $Q P_M=P_MQ$.

Suppose that $Q[X]$ is moreover separable. Since $Q[X]\in M$ by Lemma~\ref{l:predp}(ii), there is a countable dense set $C\subset Q[X]$ such that $C\in M$ (by Lemma~\ref{l:predp}(xii)), hence $C\subset M$ (by Lemma~\ref{l:predp}(v)). It follows that $Q[X]\cap M$ is dense in $Q[X]$, hence $Q[X]\subset \overline{X\cap M}=P_M[X]$. Therefore $P_MQ=Q$ which completes the proof.
\end{proof}

\begin{lem}\label{le:kom2} For a suitable elementary model $M$ the following holds: Let $\m$ be a \JBWs{} and $(p_\alpha)_{\alpha\in\Gamma}$ an orthogonal system of $\sigma$-finite projections in $\m$ with sum equal to $1$. Let $\Lambda_0$ and $p_C$, $R_C$, $C\in\Lambda_0$ be defined as in Proposition~\ref{P:skel2}. Set $D=\bigcup_{C\in\Lambda_0} R_C^*[\m]$. For any $C\in\Lambda_0$ let  $(S_{C,j})_{j\in J_C}$  be a commutative projectional skeleton in $R_C[\m_*]$. Suppose that $M$ contains $\m$, $\m_*$, $D$, $(p_\alpha)_{\alpha\in\Gamma}$, $(R_C)_{C\in\Lambda_0}$ and $((S_{C,j})_{j\in J_C})_{C\in\Lambda_0}$.
Denote by $P_M$ the projection induced by $M$. Then the following
assertions are fulfilled:
\begin{itemize}

	\item[(a)] $P_M$ commutes with $R_C$ for each $C\subset \Gamma\cap M$.

	\item[(b)] For any $C\in \Lambda_0\cap M$ there is $j_C\in J_C$ such that $P_M$ restricted to $R_C[\m_*]$ equals $S_{C,j_C}$.

	\item[(c)] Let $C=\Gamma\cap M$. Then $P_M R_C=R_C P_M=P_M$.

	\end{itemize}
\end{lem}

\begin{proof}
 Let $\phi_1,\dots,\phi_N$ be a subformula-closed list of formulas which contains the formulas from Lemma~\ref{l:predp}, the formulas provided by Lemmata~\ref{le:projekce} and~\ref{le:wld} and the formulas below marked by $(*)$. Let $Y$ be a countable subset containing the set $Y_0$ from Lemma~\ref{l:predp} and the sets provided by Lemmata~\ref{le:projekce} and~\ref{le:kom1}. Fix an arbitrary $M\prec(\phi_1,\dots,\phi_N;Y)$ containing $\m$, $\m_*$, $D$, $(p_\alpha)_{\alpha\in\Gamma}$, $(R_C)_{C\in\Lambda_0}$ and $((S_{C,j})_{j\in J_C})_{C\in\Lambda_0}$.

Fix any $C\subset \Gamma\cap M$. For any finite subset $F\subset C$ we get $F\in M$ by Lemma~\ref{l:predp}(iii). Then $R_F\in M$ by Lemma~\ref{l:predp}(ii). Therefore by Lemma~\ref{le:kom1} we deduce that $R_F$ commutes with $P_M$. Since $R_C$ is the SOT-limit of these projections $R_F$, we conclude that $R_C$ commutes with $P_M$ as well. This completes the proof of the assertion (a).

Let us continue by proving (b). Fix $C\in\Lambda_0\cap M$. Then $C$ is a countable subset of $\Gamma$, thus $C\subset \Gamma\cap M$ by Lemma~\ref{l:predp}(v). By (a) it follows that $P_M$ commutes with $R_C$. In particular, $P_M$ restricted to $R_C[\m_*]$
is a projection on $R_C[\m_*]$. Further, since $C\in M$, we get $(S_{C,j})_{j\in J_C}\in M$, hence also $J_C\in M$ (we apply Lemma~\ref{l:predp}(ii) twice).

It follows by Lemma~\ref{l:predp}(xiii) that $J_C\cap M$ is a countable up-directed set, denote by $j_C$ its supremum. 
For any $j\in J_C\cap M$ we have
$P_M S_{C,j}=S_{C,j} P_M=S_{C,j}$ by Lemma~\ref{le:kom1}. Hence, by proceeding to the SOT-limit we get
$$P_M S_{C,j_C}=S_{C,j_C} P_M=S_{C,j_C}.$$
To complete the proof of (b) it suffices to observe that the range of $P_M R_C$ is contained in the range of $S_{C,j_C}$.
But
$$P_M[R_C[\m_*]]=R_C[P_M[\m_*]]=R_C[\overline{\m_*\cap M}]\subset \overline{R_C[\m_*\cap M]}$$
and for any $\omega\in\m_*\cap M$ we have $R_C(\omega)\in R_C[\m_*]\cap M$. Since
$$\exists j\in J_C: S_{C,j}\omega=\omega,\eqno{(*)}$$
elementarity yields such a $j\in J_C\cap M$. Therefore $S_{C,j_C}\omega=\omega$.

Finally, let us prove (c). The first equality follows from (a). To complete the proof it is enough to show that
the range of $P_M$ is contained in the range of $R_C$. Since the range of $P_M$ is $\overline{\m_*\cap M}$, it suffices to observe that $\m_*\cap M\subset R_C[\m_*]$. But this can be proved by repeating the argument from the proof of (b).
\end{proof}

\begin{proof}[Proof of Theorem~\ref{T:main}] We start by proving the theorem for \JBWs{}s. To this end we will use Lemma~\ref{le:plichko}.
Let $\m$ be any \JBWs{} and let $(p_\alpha)_{\alpha\in\Gamma}$, $\Lambda_0$ and $p_C$, $R_C$, $C\in\Lambda_0$ be defined as in Proposition~\ref{P:skel2}. Set $D=\bigcup_{C\in\Lambda_0} R_C^*\m$. For any $C\in\Lambda_0$ let  $(S_{C,j})_{j\in J_C}$  be a commutative projectional skeleton in $R_C\m_*$.

 Let $\phi_1,\dots,\phi_N$ be a subformula-closed list of formulas which contains the formulas from Lemma~\ref{l:predp}, the formulas provided by Lemmata~\ref{le:projekce} and~\ref{le:wld}. Let $Y$ be a countable subset containing the set $Y_0$ from Lemma~\ref{l:predp} and the sets provided by Lemmata~\ref{le:projekce} and~\ref{le:kom1}
and containing also $\m$, $\m_*$, $D$, $(p_\alpha)_{\alpha\in\Gamma}$, $(R_C)_{C\in\Lambda_0}$ and $((S_{C,j})_{j\in J_C})_{C\in\Lambda_0}$.
Let $M_1$ and $M_2$ be two elementary models for $\phi_1,\dots,\phi_N$ containing $Y$.

Let $C_1=M_1\cap \Gamma$, $C_2=M_2\cap \Gamma$ and $C=C_1\cap C_2$. Let $C=\{\gamma_n;n\in\N\}$ and $F_n=\{\gamma_1,\dots,\gamma_n\}$. Since $C\subset M_1\cap M_2$ and $F_n$ is finite, we get $F_n\in M_1\cap M_2$
for each $n$ (by Lemma~\ref{l:predp}(iii)). Therefore, by Lemma~\ref{le:kom2} we find $j_n,k_n\in J_{F_n}$ such that
$$P_{M_1}|_{R_{F_n}[\m_*]}=S_{F_n,j_n}\quad\mbox{and}\quad P_{M_2}|_{R_{F_n}[\m_*]}=S_{F_n,k_n}.$$
Fix any $\omega\in\m_*$. We have
$$ \begin{aligned} P_{M_1}P_{M_2}\omega
&= (P_{M_1}R_{C_1})(P_{M_2}R_{C_2})\omega
=P_{M_1}R_{C_1}R_{C_2}P_{M_2}\omega
= P_{M_1}R_{C}P_{M_2}\omega
\\&= P_{M_1}P_{M_2}R_{C}\omega
=\lim_n P_{M_1}P_{M_2}R_{F_n}\omega
=\lim_n P_{M_1}S_{F_n,k_n}R_{F_n}\omega
\\&=\lim_n S_{F_n,j_n}S_{F_n,k_n}R_{F_n}\omega.\end{aligned}$$
Similarly we get
$$ P_{M_2}P_{M_1}\omega
=\lim_n S_{F_n,k_n}S_{F_n,j_n}R_{F_n}\omega.$$
Since the projections $S_{F_n,k_n}$ and $S_{F_n,j_n}$ commute, we conclude
that $P_{M_1}$ and $P_{M_2}$ commute as well.

If $\m$ is $\sigma$-finite, then $\m_*$ is WCG by Proposition~\ref{L:wcg} applied to $p=1$. Next suppose that $\m$ is not $\sigma$-finite. Similarly as in the proof of \cite[Theorem 1.1]{BHK} to show that $\m_*$ is not WLD it suffices to prove that it contains an isometric copy of $\ell^1(\Gamma)$ for an uncountable set $\Gamma$. Such a set $\Gamma$ will be provided by Proposition~\ref{P:skel2} -- it is uncountable due to Lemma~\ref{L:spojitost}. For any $\alpha\in\Gamma$ let $\omega_\alpha$ be a normal state such that $p_\alpha=p_{\omega_\alpha}$ (it exists due to Lemma~\ref{L:sigmaf}). We claim that the closed linear span of $(\omega_\alpha)_{\alpha\in\Gamma}$ in $\m_*$ is isometric to $\ell^1(\Gamma)$.  To prove the claim fix a finite set $F\subset\Gamma$ and $c_\alpha\in\cc$ for $\alpha\in F$. For each $\alpha\in F$ fix a complex unit $\theta_\alpha$ such that $\theta_\alpha c_\alpha=|c_\alpha|$. 
Set $x=\sum_{\alpha\in F}\theta_\alpha p_{\alpha}$. Then $x^*=\sum_{\alpha\in F}\overline{\theta_\alpha}p_{\alpha}$ and hence $x^*\circ x=\sum_{\alpha\in F} p_{\alpha}=p_F$. Hence,
$$
\begin{aligned}\{xx^*x\}&=2(x\circ x^*)\circ x-(x\circ x)\circ x^*
\\&=2\left(\left(\sum_{\alpha\in F}{\theta_\alpha}p_{\alpha}\right) \circ \left(\sum_{\alpha\in F}\overline{\theta_\alpha}p_{\alpha}\right)\right)
\circ\left(\sum_{\alpha\in F}{\theta_\alpha}p_{\alpha}\right)
\\&\qquad\qquad-\left(\left(\sum_{\alpha\in F}{\theta_\alpha}p_{\alpha}\right)\circ
\left(\sum_{\alpha\in F}{\theta_\alpha}p_{\alpha}\right)\right)\circ
\left(\sum_{\alpha\in F}\overline{\theta_\alpha}p_{\alpha}\right)
\\&=2\left(\sum_{\alpha\in F}p_{\alpha}\right)\circ\left(\sum_{\alpha\in F}{\theta_\alpha}p_{\alpha}\right)
-\left(\sum_{\alpha\in F}{\theta_\alpha^2}p_{\alpha}\right)
\circ
\left(\sum_{\alpha\in F}\overline{\theta_\alpha}p_{\alpha}\right)
\\&=2\left(\sum_{\alpha\in F}{\theta_\alpha}p_{\alpha}\right)-\left(\sum_{\alpha\in F}{\theta_\alpha}p_{\alpha}\right)
=2x-x=x.
\end{aligned}$$
So, $\|x\|^3=\|\{xx^*x\}\|=\|x\|$, hence $\|x\|=1$ (unless the trivial case $F=\emptyset$). Further,
$$\left\|\sum_{\alpha\in F} c_\alpha\omega_{p_\alpha}\right\|\ge \left|\sum_{\alpha\in F} c_\alpha\omega_{p_\alpha}(x)\right|=\sum_{\alpha\in F}|c_\alpha|.$$
Since the converse inequality is obvious, we conclude that $\|\sum_{\alpha\in F} c_\alpha\omega_{p_\alpha}\|=\sum_{\alpha\in F}|c_\alpha|$, which completes the proof.

Finally, let us prove the theorem in case of \JBW{}s. Let $\A$ be a \JBW. By Lemma~\ref{L:JB-JB*} there is a unique $JB^*$-algebra $\m$ such that the $\A$ is isometrically isomorphic to $\saM$. By \cite[Theorem 3.4]{edwards} $\m$ is a \JBWs. Moreover, $\A_*$ is isometric to $\psaM$ by Lemma~\ref{L:JBW-JBW*}, hence it is enough to prove the statement for $\psaM$. Let $(p_\alpha)_{\alpha\in\Gamma}$, $\Lambda_0$ and $p_C$, $R_C$, $C\in\Lambda_0$ be defined as in Proposition~\ref{P:skel2}. It follows from
Lemma~\ref{L1-Up} that the projections $R_C$ preserve $\psaM$. Define by $R^{sa}_C$ the restriction of $R_C$ to $\psaM$, considered as a projection on $\psaM$. Since $R^{sa}_C[\psaM]$ is a complemented subspace of the WCG space $R_C[\m_*]$, it is WCG as well. Hence, we can fix, for each $C\in\Lambda_0$, 
a commutative projectional skeleton $(S_{C,j})_{j\in J_C}$ in $R_C^{sa}[\psaM]$.
Using an obvious analogue of Lemma~\ref{le:kom2} for $\psaM$ we can prove
that $\psaM$ satisfies the assumptions of Lemma~\ref{le:plichko} to conclude
in the same way as in case of $\m$. The assertions on $\sigma$-finite and non-$\sigma$-finite \JBW{}s can be done in the same way as in case of \JBWs{}s.
\end{proof}


\end{document}